\documentclass[12pt]{amsart}
\usepackage[utf8]{inputenc}
\usepackage[a4paper,inner=2.5cm,outer=2.5cm,top=2.5cm,bottom=2.5cm]{geometry}
\usepackage{amsmath,amssymb,amsthm}

\usepackage{hyperref}
\hypersetup{colorlinks=true,linkcolor=blue,citecolor=teal,
filecolor=magenta,urlcolor=cyan}

\usepackage{tikz}
\usetikzlibrary{calc,decorations.markings,arrows.meta}
\tikzset{baseline={([yshift=-.5ex]current bounding box.center)}}
\tikzstyle directed=[blue,postaction={decorate,decoration={markings, mark= at position .55 with {\arrow{stealth}}}}]
\tikzstyle{v} = [draw,black,fill,circle, inner sep=1pt]
\tikzstyle{->}=[-{Stealth[]}]
\tikzstyle{<->}=[{Stealth[]}-{Stealth[]}]
\tikzstyle{<-}=[{Stealth[]}-]

\usepackage{xcolor}
\newcommand{\mkaz}[1]{{\textcolor{blue}{#1}}}

\newcommand{\zyan}[1]{{\textcolor{red}{#1}}}

\newtheorem{theorem}{Theorem}[section]

\newtheorem{corollary}[theorem]{Corollary}
\newtheorem{proposition}[theorem]{Proposition}

\theoremstyle{definition}
\newtheorem{definition}[theorem]{Definition}
\newtheorem{remark}[theorem]{Remark}

\newcommand{\C}{{\mathbb{C}}}
\newcommand{\Q}{{\mathbb{Q}}}
\newcommand{\Z}{{\mathbb{Z}}}
\newcommand{\fg}{{\mathfrak{g}}}
\newcommand{\fgl}{{\mathfrak{gl}}}
\newcommand{\fsl}{{\mathfrak{sl}}}
\newcommand{\fso}{{\mathfrak{so}}}
\newcommand{\fsp}{{\mathfrak{sp}}}
\newcommand{\fosp}{{\mathfrak{osp}}}
\newcommand{\fp}{{\mathfrak{p}}}
\newcommand{\fq}{{\mathfrak{q}}}

\DeclareMathOperator{\Tr}{Tr}

\title{Universal Polynomial $\fso$ Weight System}

\author[M.~Kazarian]{Maxim~Kazarian}
\address{M.~K.: Faculty of Mathematics, National Research University Higher School of Economics, Usacheva 6, 119048 Moscow, Russia; and Center for Advanced Studies, Skoltech, Nobelya 1, 143026, Moscow, Russia}
\email{kazarian@mccme.ru}

\author[Zh.~Yang]{Zhuoke Yang}
\address{Zh.~Y.: Beijing Institute of Mathematical Sciences and Applications, 101407, Beijing, China}
\email{yangzhuoke@bimsa.cn}

\date{January 2023}

\begin{document}

\maketitle

\begin{abstract}
We introduce a universal weight system (a function on chord diagrams satisfying the $4$-term relation) taking values in the ring of polynomials in infinitely many variables whose particular specializations are weight systems associated with the Lie algebras $\fso(N)$, $\fsp(2M)$, as well as Lie superalgebras $\fosp(N|2M)$. We extend this weight system to permutations and provide an efficient recursion for its computation.

The construction for this weight system extends a similar construction for the universal polynomial weight system responsible for the Lie algebras $\fgl(N)$ and superalgebras $\fgl(N|M)$ introduced earlier by the second named author.
\end{abstract}

\section{Introduction}

\subsection{Weight systems and Vassiliev knot invariants}
Below, we use standard notions from the theory of finite order knot invariants;
see, e.g.~\cite{Chmutov2012,lando2013graphs}.

A {\it chord diagram\/} of order~$m$ is an oriented circle (called the \emph{Wilson loop}) endowed with~$2m$
pairwise distinct points split into~$m$ disjoint pairs, considered up to
orientation-preserving diffeomorphisms of the circle. These pairs of points are usually depicted as chords.

A~{\it weight system\/} is a function $w$ on chord diagrams satisfying the $4$-term
relations; see Fig.~\ref{fourtermrelation}.
\def\chord(#1,#2);{
\coordinate (->) at ($(0,0)!1!#1:(1,0)$);
\coordinate (b) at ($(0,0)!1!#2:(1,0)$);
\draw (->) .. controls ($0.35*(->) + 0.35*(b)$) .. (b);}
\newcommand{\Tfour}[1]{\begin{tikzpicture}[scale=.8]
\draw[dashed] (0,0) circle [radius=1];
\draw[thick] (30:1) arc [start angle=30, end angle=60, radius=1]
        (120:1) arc [start angle=120, end angle=150, radius=1]
        (260:1) arc [start angle=260, end angle=290, radius=1];
#1
\end{tikzpicture}}
\begin{figure}[ht]
$$
    w\Biggl(~\Tfour{\chord(45,269); \chord(135,281);}~\Biggr)
   -w\Biggl(~\Tfour{\chord(45,280); \chord(135,270);}~\Biggr)=
   w\Biggl(~\Tfour{\chord(51,275); \chord(135,39);}~\Biggr)-
   w\Biggl(~\Tfour{\chord(40,275); \chord(135,50);}~\Biggr)
$$
\caption{$4$-term relations}
\label{fourtermrelation}
\end{figure}

In figures, the outer circle of the chord diagram is always assumed to be
oriented counterclockwise. Dashed arcs may contain ends of arbitrary sets of chords, same for all the four terms in the picture. Weight systems play a crucial role in the construction of Vassiliev knot invariants.

\subsection{Universal polynomial $\fgl$ weight system}

There is a general construction for a weight system $w_\fg$ associated with an arbitrary Lie algebra~$\fg$ equipped with an ${\rm ad}$-invariant nondegenerate symmetric bilinear form~\cite{kon1993,bar1995vassiliev}. It takes values in the center on the universal enveloping algebra $U\fg$. For the case $\fg=\fgl(N)$, this center is freely generated by the Casimir elements $C_{\fgl(N),1},\dots,C_{\fgl(N),N}$, where
\begin{equation}\label{eq:glCasimir}
C_{\fgl(N),k}=\sum\limits_{i_1,i_2,\dots, i_k=1}^{N}E_{i_1i_2}E_{i_2i_3}\dots E_{i_ki_1}.
\end{equation}
Here $E_{i\,j}\in\fgl(N)$ is the matrix unit having~$1$ on the intersection of the $i$th row and the $j$th column and~$0$ elsewhere, $i,j=1,\dots,N$.

Thus, $w_{\fgl(N)}$ takes values in the commutative ring $\C[C_{\fgl(N),1},\dots,C_{\fgl(N),N}]$ of polynomials in $N$ independent variables. These weight systems are  quite strong and have a number of important applications. However, the direct computation of this weight system from its definition is rather laborious, since one has to deal with a complicated noncommutative algebra~$U\fgl(N)$. In addition, for a fixed chord diagram $D$ the complexity of computation of~$w_{\fgl(N)}(D)$ grows very rapidly with the growth of~$N$. This difficulty was overcome in~\cite{Y1} by presenting a construction for a universal polynomial weight system, denoted by~$w_\fgl$, possessing the following properties (we recall the definition of~$w_\fgl$ below in the next section):
\begin{itemize}
\item $w_\fgl$ takes values in the ring $\C[C_0,C_1,\dots]$ of polynomials in infinitely many generators. For a given chord diagram~$D$ with $k$ chords, the polynomial $w_{\fgl}(D)$ involves the variables $C_0,\dots,C_k$ only;

\item $w_\fgl$ is a weight system indeed, that is, it does satisfy the $4$-term relations;

\item for any $N\ge1$, the weight system~$w_{\fgl(N)}$ can be obtained from $w_\fgl$ by the specialization taking~$C_0$ to~$N$ and~$C_k$ to the corresponding Casimir element $C_{\fgl(N),k}$ for $k\ge1$.
\end{itemize}

\begin{remark}
The specialization taking~$w_\fgl$ to $w_{\fgl(N)}$ may involve the Casimir elements $C_{\fgl(N),k}$ defined by~\eqref{eq:glCasimir} for $k>N$. These elements also belong to the center of~$U\fgl(N)$, and they can be expressed as polynomials in $C_{\fgl(N),1},\dots,C_{\fgl(N),N}$. These polynomial expressions for higher Casimirs can be obtained from the following formula due to Perelomov and Popov, which identifies the ring $ZU\fgl(N)=\C[C_{\fgl(N),1},\dots,C_{\fgl(N),N}]$ with the ring of symmetric functions in the indeterminates $x_1,\dots,x_N$, see~\cite{Perelomov-Popov} and~\cite[\S60]{Zh}:
$$
1-N\,u-\sum_{k=1}^\infty C_{\fgl(N),k}\,u^{k+1}=\prod_{i=1}^N\left(1-\frac{1}{1-x_iu}\right).
$$
Equating the first $N$ coefficients in these generating series allows one to identify $C_{\fgl(N),1},\dots,C_{\fgl(N),N}$ as the generators of the ring of symmetric functions while the remaining terms of this equality can be used to express $C_{\fgl(N),k}$ in terms of these generators for~$k>N$.
\end{remark}

Existence of~$w_\fgl$ implies that for a fixed chord diagram~$D$ with $k$ chords and $N>k$ the invariant $w_{\fgl(N)}(D)$ written as a polynomial in the Casimir elements $C_i=C_{\fgl(N),i}$ involves the generators $C_1,\dots,C_k$ only, and moreover, the coefficients of this polynomial depend on~$N$ in a polynomial way. This fact is not obvious at all from the definition of~$w_{\fgl(N)}$.

The weight system associated with the Lie algebra~$\fsl(N)$ can also be obtained from the same universal one $w_\fgl$ by setting additionally $C_1=0$, which is not surprising, see~\cite{Y1},~\cite{KL}. What is much more surprising is that the same universal weight system~$w_\fgl$ contains also the values of the weight system associated with the Lie superalgebras $\fgl(N|M)$ for all~$N$ and~$M$. A construction of the weight system associated with a metrized Lie superalgebra is parallel to the one associated with a metrized Lie algebra, see~\cite{Chmutov2012,FKV97}. The main result of~\cite{Y2} claims that for the case of the Lie superalgebra~$\fgl(N|M)$ the weight system $w_{\fgl(N|M)}$ can be obtained from $w_\fgl$ by the specialization taking $C_0$ to $N-M$ and $C_k$ to the corresponding Casimir elements. The Casimir elements of $\fgl(N|M)$ generate the center of~$U\fgl(N|M)$, and relations among them are implied by the following analog of the Perelomov-Popov formula, which identifies the center of $U\fgl(N|M)$ with the ring of supersymmetric functions in the indeterminates $x_1,\dots,x_N,y_1,\dots,y_M$:
$$
1-(N-M)\,u-\sum_{k=1}^\infty C_{\fgl(N|M),k}\,u^{k+1}=
\prod_{i=1}^N\frac{1-(1+x_i)\,u}{1-x_i u}.
\prod_{j=1}^M\frac{1+y_j\,u}{1-(1-y_j)u}.
$$

 The goal of this paper is to find  the analogues of the weight system~$w_\fgl$ for other classical series of Lie algebras and superalgebras, and to establish relationship between these universal polynomial invariants.

 \subsection{Defining relations for the universal polynomial $\fgl$ weight system}

 The weight system~$w_\fgl$ is defined in~\cite{Y1} through a number of relations leading to its explicit recursive computation. These relations exit, in fact, the set of chord diagrams, and $w_\fgl$ is actually defined as a \emph{function on the set of permutations} (of arbitrary number of elements). In particular, a chord diagram with~$M$ chords is interpret as an involution without fixed points on the $2M$ element set of the ends of the chords. The permutation in question swaps the ends of the chords; the points being ordered counterclockwise starting from arbitrarily chosen cut point on the circle (the value of~$w_\fgl$ is independent of the choice of the cut point).

 The definition of the invariant $w_\fgl$ presented below is motivated by the requirement that its specialization to $w_{\fgl(N)}$ on a given permutation~$s\in S(m)$ should be given by
 \begin{equation}\label{eq:wglNsigma}
	w_{\fgl(N)}(s)=\sum_{i_1,\cdots,i_m=1}^N E_{i_1i_{s(1)}}E_{i_2i_{s(2)}}\cdots E_{i_mi_{s(m)}}\in U\fgl(N).
\end{equation}

In order to formulate defining relations for~$w_\fgl$ it would be convenient to represent permutations by graphs.

\begin{definition}[digraph of the permutation] Let us represent permutations by directed graphs (digraphs). The digraph $\sigma=\sigma(s)$ of a permutation $s\in S(m)$ consists of~$M$ numbered vertices $1,2,\dots,m$, and~$m$ edges $i\to s(i)$ connecting vertices $i$ and $s(i)$ and directed from $i$ to $s(i)$ for all $i=1,\dots,m$. 
\end{definition}

Instead of numbering vertices it is convenient to place them on an additional oriented line in the order compatible with the orientation of the line, for example,
$$
s=(35214)\quad\longleftrightarrow\quad
\sigma(s)=
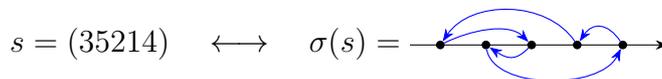
\begin{tikzpicture}
\draw (0,0) node[v] (V1) {}  (0.6,0) node[v] (V2) {} (1.2,0) node[v] (V3) {} (1.8,0) node[v] (V4) {} (2.4,0) node[v] (V5) {};
 \draw[->] (-.4,0) -- (3.0,0);
 \draw[->,blue] (V1) .. controls +(.4,.2) and +(-.3,.3) .. (V3);
 \draw[->,blue] (V2) .. controls +(.1,-.6) and +(-.3,-.6) .. (V5);
 \draw[->,blue] (V3) .. controls +(-.2,-.2) and +(.3,-.2) .. (V2);
 \draw[->,blue] (V4) .. controls +(-.4,.6) and +(.4,.6) .. (V1);
 \draw[->,blue] (V5) .. controls +(-.2,.3) and +(.2,.3) .. (V4);
 \end{tikzpicture}
$$

For a given chord diagram the corresponding digraph is obtained by replacing every chord by an oriented two-cycle (with the oriented line contained the vertices represented by the Wilson loop cut at same point).

Another example is the digraph associated with the Casimir element $C_m$ which corresponds to the standard cyclic permutation $1\mapsto2\mapsto\cdots\mapsto m\mapsto1\in S(m)$ and which we denote also by $C_m$. It is a length~$m$ oriented cycle whose orientation is compatible with the cyclic ordering of vertices:
$$
C_m=~
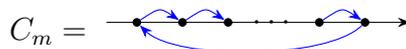
\begin{tikzpicture}
\draw (0,0) node[v] (V1) {}  (0.6,0) node[v] (V2) {} (1.2,0) node[v] (V3) {} (2.4,0) node[v] (V4) {} (3,0) node[v] (V5) {}
(1.8,0) node {$\cdots$};
\draw[->] (-0.4,0) -- (3.6,0);
 \draw[->,blue] (V1) .. controls +(.3,.2) and +(-.3,.2) ..  (V2);
 \draw[->,blue] (V2) .. controls +(.3,.2) and +(-.3,.2) .. (V3);
 \draw[->,blue] (V4) .. controls +(.3,.2) and +(-.3,.2) ..  (V5);
 \draw[->,blue] (V5) .. controls +(-.5,-.5) and +(.5,-.5) .. (V1);
 \end{tikzpicture}
$$

We are ready to formulate the defining relations for the $w_\fgl$ polynomial invariant.

\begin{definition}
The $w_\fgl$ invariant is a function on the set of all permutations (of any finite number of ordered elements) taking values in the ring $\C[C_0,C_1,\dots]$ of polynomials in an infinite number of independent generators and satisfying the following properties (axioms):
\begin{itemize}
		\item $w_{\fgl}$ is multiplicative with respect to concatenation of permutations (the concatenation of permutations corresponds to the disjoint union of their permutation graphs with the ordering on the set of vertices obtained by the linear order on the set of vertices of the first graph followed by the linear order on the set of vertices of the second graph). As a corollary, for the empty graph (with no vertices) the value of $w_{\fgl}$ is equal to $1$;
	\item for the standard cyclic permutation {\rm(}with the cyclic order on the set of permuted elements  compatible with the permutation{\rm)},
 the value of $w_{\fgl}$ is the standard generator,
		$w_\fgl(1\mapsto2\mapsto\cdots\mapsto m\mapsto1)=C_m$;

\item {\rm(}\textbf{Recurrence Rule}{\rm)} For the graph of an arbitrary permutation $s$,
and for any pair of its vertices labelled by consecutive integers $r,r+1$, we have for the values of the $w_\fgl$ weight system
\begin{equation}\label{eq:rec-gl}
w_\fgl\left(~
\begin{tikzpicture}[scale=0.8]
\draw (0.6,0) node[v] (V1) {} (1.2,0) node[v] (V2) {};
\draw  (0.5,-.25) node {\scriptsize$r$}  
 (1.3,-.25) node {\scriptsize$r{+}1$};
\draw [->] (0,0) -- (1.8,0);
\draw[->,blue] (0,0.7) .. controls (0.6,0.5) .. (V2);
\draw[->,blue] (V2) .. controls (0.6,0.9) .. (0,1.4);
\draw[->,blue] (1.8,1.4) .. controls (1.2,0.9) .. (V1);
\draw[->,blue] (V1) .. controls (1.2,0.5) .. (1.8,0.7);
\end{tikzpicture}
 ~\right)=w_\fgl\left(~
\begin{tikzpicture}[scale=0.8]
\draw  (0.7,0) node[v] (V1) {} (1.1,0) node[v] (V2) {};
\draw [->] (0,0) -- (1.8,0);
\draw[->,blue] (0,0.7) .. controls (0.4,0.4) .. (V1);
\draw[->,blue] (V1) .. controls (0.4,0.9) .. (0,1.4);
\draw[->,blue] (1.8,1.4) .. controls (1.4,0.9) .. (V2);
\draw[->,blue] (V2) .. controls (1.4,0.4) .. (1.8,0.7);
\end{tikzpicture}
 ~\right)+w_\fgl\left(~
\begin{tikzpicture}[scale=0.8]
\draw (0.9,0) node[v] (V) {};
\draw [->] (0,0) -- (1.8,0);
\draw[->,blue] (0,0.7) .. controls +(0.5,-0.5) and +(-0.5,-0.5) .. (1.8,0.7);
\draw[->,blue] (1.8,1.4) .. controls +(-0.6,-0.7) .. (V);
\draw[->,blue] (V) .. controls +(-0.3,0.7) .. (0,1.4);
\end{tikzpicture}
 ~\right)-w_\fgl\left(~
\begin{tikzpicture}[scale=0.8]
\draw (0.9,0) node[v] (V) {};
\draw [->] (0,0) -- (1.8,0);
\draw[->,blue] (0,0.7) .. controls (0.4,0.5) .. (V);
\draw[->,blue] (V) .. controls (1.4,0.5) .. (1.8,0.7);
\draw[->,blue] (1.8,1.4) .. controls (0.9,0.2) .. (0,1.4);
\end{tikzpicture}
 ~\right)
\end{equation}
The first digraph on the right differs from the one on the left only by swapping the order of the two neighbouring vertices. For the last two digraphs, we switch the positions of the heads of two edges and remove one of the two vertices joining the corresponding edges. These two graphs have one smaller number of vertices so we renumber them by preserving the relative order implied by their position on the horizontal line.

A closed cycle with no vertices appearing in the graphs on the right hand side in the case $s(r+1)=r$ is interpreted by convention as a factor~$C_0$. More explicitly, for this exceptional case we have
\begin{equation}\label{eq:rec-gl2}
w_\fgl\left(~
\begin{tikzpicture}[scale=0.8]
\draw (0.6,0) node[v] (V1) {} (1.2,0) node[v] (V2) {};
\draw [->] (0,0) -- (1.8,0);
\draw[->,blue] (0,1) .. controls (0.7,0.6) .. (V2);
\draw[->,blue] (V1) .. controls (1.1,0.6) .. (1.8,1);
\draw[->,blue] (V2) .. controls (0.9,0.2) .. (V1);
\end{tikzpicture}
 ~\right)=
 w_\fgl\left(~
\begin{tikzpicture}[scale=0.8]
\draw  (0.6,0) node[v] (V1) {} (1.2,0) node[v] (V2) {};
\draw [->] (0,0) -- (1.8,0);
\draw[->,blue] (0,1) .. controls (0.4,0.6) .. (V1);
\draw[->,blue] (V2) .. controls (1.4,0.6) .. (1.8,1);
\draw[->,blue] (V1) .. controls (0.9,0.3) .. (V2);
\end{tikzpicture}
 ~\right)+
C_0\,w_\fgl\left(~
\begin{tikzpicture}[scale=0.8]
\draw (0.9,0) node[v] (V) {};
\draw [->] (0,0) -- (1.8,0);
\draw[->,blue] (V) .. controls (1.2,0.5) .. (1.8,1);
\draw[->,blue] (0,1) .. controls (0.6,0.5) .. (V);
\end{tikzpicture}
 ~\right)-
  C_1\,w_\fgl\left(~
\begin{tikzpicture}[scale=0.8]
\draw [->] (0,0) -- (1.8,0);
\draw[->,blue] (0,1) .. controls (0.9,0) .. (1.8,1);
\end{tikzpicture}
 ~\right)
\end{equation}
\end{itemize}
\end{definition}

These relations are indeed a recursion, that is, they allow one to
replace the computation of $w_{\fgl}$ on a permutation with its computation on simpler permutations. Indeed, by renumbering vertices any digraph can be reduced to the product of independent standard cycles, so that the value of $w_\fgl$ on the permutation is equal to the corresponding monomial in $C_k$'s plus some combination of values of this invariant on permutations with smaller number of elements.

Compatibility of this invariant with the values of $w_{\fgl(N)}$ defined by~\eqref{eq:wglNsigma} follows from the commutation relations in~$\fgl(N)$, see details in~\cite{Y1}. The correctness of the definition of the $w_\fgl$ invariant (i.e. the fact that the result is independent of the order in which the recursion relations are applied) follows from the fact that for a given permutation~$s$ the polynomial $w_\fgl(s)$ can be recovered from the values of~$w_{\fgl(N)}(s)$ for sufficiently large~$N$.

Note that the fact that the element~\eqref{eq:wglNsigma} belongs to the center of~$U\fgl(N)$ follows from the recursion: we express this element as a polynomial in the Casimirs. Another corollary of the recursion is the cyclic invariance of~$w_\fgl$: conjugation of a permutation by the standard cyclic permutation preserves this invariant. In other terms, the value of the~$w_\fgl$ weight system on a given permutation depends on the cyclic order of the permuted elements only. We provide below ~\ref{thm:cyclic} an independent proof of this fact.

\subsection{Results of the paper}
The paper is organized as follows. In Sec.~2, we describe the universal $\fso$ weight system on arbitrary permutations as an extension of the weight systems $w_{\fso(N)},w_{\fsp(2M)},w_{\fosp(N|2M)}$ and a recursion to computing its values on permutations. We give a proof of cyclic invariance of~$w_\fgl$ and ~$w_\fso$ in the end.
In Sec.~3, we prove that the three weight systems $w_{\fso(N)},w_{\fsp(2M)},w_{\fosp(N|2M)}$ satisfy the very same recursion of the universal $\fso$ weight system.
In Sec.~4, using the recursion relations for the $w_\fso$, we calculate the value of $w_\fso$ on small order chord diagrams. We show that the universal weight system $w_\fso$, while being much weaker than the $w_{\fgl}$ weight system, is however independent and cannot be induced from $w_\fgl$. 
In the Appendix, we recall the Perelomov-Popov formula for Casimirs.

The authors are grateful to S.~Lando for his constant attention to their work. This work was supported by the International Laboratory of Cluster Geometry NRU HSE, RF Government grant, ag. № 075-15-2021-608 dated 08.06.2021.

\section{Construction of the universal $\fso$ weight system}

\subsection{Definition of $w_\fso$}

We define in this section a multiplicative weight system $w_\fso$ taking values in the ring of polynomials in the generators $C_0,C_2,C_4,C_6,\dots$ labelled by even nonnegative integers. Similarly to the case of the $w_\fgl$ weight system, we extend the definition of $w_\fso$ to the set of permutations (on any number of permuted elements). This extension is defined by a set of relations that are close to those for the $w_\fgl$ case. The defining relations presented below are motivated by the requirement that for the Lie algebra $\fso(N)$ (with $N$ of any parity) the specialization of $w_\fso$ taking $C_k$ to the corresponding  Casimir element is given by
 \begin{equation}\label{eq:wsonsigma}
	w_{\fso(N)}(s)=\sum_{i_1,\cdots,i_m=1}^N X_{i_1i_{s(1)}}X_{i_2i_{s(2)}}\cdots X_{i_mi_{s(m)}}\in U\fso(N),
\end{equation}
where $X_{ij}$ are the standard generators of $\fso(N)$, see Appendix.
We show in the next section that the weight systems associated with the Lie algebras $\fsp(2M)$ as well as the Lie superalgebras $\fosp(N|2M)$ are also specializations of the universal weight system $w_\fso$.
 
Before formulating defining relations for $w_\fso$, we introduce some notation used in these relations. 

The invariant $w_\fso$ constructed below possesses an additional symmetry that does not hold for the $\fgl$ case:
\begin{quote}
\emph{assume that the permutation $s'$ is obtained from $s$ by the inversion of one of its independent cycles. In this case the values $w_\fso(s)$ and $w_\fso(s')$ differ by the sign factor $(-1)^r$ where $r$ is the length of the cycle.}
    \end{quote}

This symmetry leads to the following convention. Along with the digraphs of permutations we consider more general graphs which we call extended permutation graphs.

\begin{definition} An \emph{extended permutation graph} is a graph with the following properties:
\begin{itemize}
\item the set of vertices of the graph is linearly ordered, which is depicted by placing them on an additional oriented line in the order compatible with the orientation of that line;
\item each vertex has valency~$2$, in particular, the number of edges is equal to the number of vertices;
\item for each half-edge it is specified whether it is a \emph{head} (marked with an arrow) or a \emph{tail}. For two half-edges adjacent to every vertex one of them is a tail and the other is a head. However, we allow for the edges of the graph to have two heads, or two tails, or a head and a tail.
\end{itemize}
\end{definition}

We extend $w_\fso$ to graphs of this kind by the following additional relation: a change of the tail and the head for the half-edges adjacent to any vertex results in the multiplication of the invariant by $-1$:
$$
w_\fso\left(~
\begin{tikzpicture}[scale=0.8]
\draw (0.9,0) node[v] (V) {};
\draw [->] (0,0) -- (1.8,0);
\draw[->,blue] (1.8,1) .. controls +(-0.6,-0.5) .. (V);
\draw[blue] (V) .. controls +(-0.3,0.5) .. (0,1);
\end{tikzpicture}
~\right)=-w_\fso\left(~
\begin{tikzpicture}[scale=0.8]
\draw (0.9,0) node[v] (V) {};
\draw [->] (0,0) -- (1.8,0);
\draw[blue] (1.8,1) .. controls +(-0.6,-0.5) .. (V);
\draw[->,blue] (0,1) .. controls (0.6,0.5) .. (V);
\end{tikzpicture}
~\right).
$$
By applying the transformation of this relation several times every extended permutation graph can be reduced to a usual permutation graph (characterized by an additional property that every edge has one head and one tail). This permutation graph is not unique but the symmetry of $w_\fso$ formulated above implies that this ambiguity does not affect the extension of $w_\fso$ to the specified set of graphs.

\begin{definition}\label{def:so}
The universal polynomial invariant $w_\fso$ is the function on the set of permutations of any number of elements (or, equivalently, on the set permutation graphs) taking values in the ring of polynomials in the generators $C_0,C_2,C_4,\dots$ and defined by the following set of relations (axioms)
\begin{itemize}
\item $w_{\fso}$ is multiplicative with respect to concatenation of permutation graphs. As a corollary, for the empty graph (with no vertices) the value of $w_{\fso}$ is equal to $1$;
\item A change of orientation of any cycle of length~$r$ in the graph results in multiplication of the value of the invariant~$w_\fso$ by~$(-1)^r$.
\item for a cyclic permutation of even number of elements {\rm(}with the cyclic order on the set of permuted elements  compatible with the permutation{\rm)} $1\mapsto2\mapsto\cdots\mapsto m\mapsto1$,
 the value of $w_{\fso}$ is the standard generator,
$$w_\fso\left(
  \begin{tikzpicture}
\draw (0,0) node[v] (V1) {}  node[above] {$\scriptsize 1$} (0.6,0) node[v] (V2) {}  node[above] {$\scriptsize 2$} (1.2,0) node[v] (V3) {} (2.4,0) node[v] (V4) {} (3,0) node[v] (V5) {}  node[above] {$\scriptsize m$}
(1.8,0) node {$\cdots$};
\draw[->] (-0.4,0) -- (3.6,0);
 \draw[->,blue] (V1) .. controls +(.3,.2) and +(-.3,.2) ..  (V2);
 \draw[->,blue] (V2) .. controls +(.3,.2) and +(-.3,.2) .. (V3);
 \draw[->,blue] (V4) .. controls +(.3,.2) and +(-.3,.2) ..  (V5);
 \draw[->,blue] (V5) .. controls +(-.5,-.5) and +(.5,-.5) .. (V1);
 \end{tikzpicture}
 \right)=C_m,\qquad m\text{ is even}.
 $$
\item {\rm(}\textbf{Recurrence Rule}{\rm)} For the graph of an arbitrary permutation $s$,
and for any pair of its vertices labelled by consecutive integers $r,r+1$, we have for the values of the $w_\fso$ weight system
\begin{multline}\label{eq:rec-so1}
w_\fso\left(~
\begin{tikzpicture}[scale=0.8]
\draw (0.6,0) node[v] (V1) {} (1.2,0) node[v] (V2) {};
\draw  (0.5,-.25) node {\scriptsize$r$}  
 (1.3,-.25) node {\scriptsize$r{+}1$};
\draw [->] (0,0) -- (1.8,0);
\draw[->,blue] (0,0.7) .. controls (0.6,0.5) .. (V2);
\draw[->,blue] (V2) .. controls (0.6,0.9) .. (0,1.4);
\draw[->,blue] (1.8,1.4) .. controls (1.2,0.9) .. (V1);
\draw[->,blue] (V1) .. controls (1.2,0.5) .. (1.8,0.7);
\end{tikzpicture}
 ~\right)=w_\fso\left(~
\begin{tikzpicture}[scale=0.8]
\draw  (0.7,0) node[v] (V1) {} (1.1,0) node[v] (V2) {};
\draw [->] (0,0) -- (1.8,0);
\draw[->,blue] (0,0.7) .. controls (0.4,0.4) .. (V1);
\draw[->,blue] (V1) .. controls (0.4,0.9) .. (0,1.4);
\draw[->,blue] (1.8,1.4) .. controls (1.4,0.9) .. (V2);
\draw[->,blue] (V2) .. controls (1.4,0.4) .. (1.8,0.7);
\end{tikzpicture}
 ~\right)+w_\fso\left(~
\begin{tikzpicture}[scale=0.8]
\draw (0.9,0) node[v] (V) {};
\draw [->] (0,0) -- (1.8,0);
\draw[->,blue] (0,0.7) .. controls +(0.5,-0.5) and +(-0.5,-0.5) .. (1.8,0.7);
\draw[->,blue] (1.8,1.4) .. controls +(-0.6,-0.7) .. (V);
\draw[->,blue] (V) .. controls +(-0.3,0.7) .. (0,1.4);
\end{tikzpicture}
 ~\right)-w_\fso\left(~
\begin{tikzpicture}[scale=0.8]
\draw (0.9,0) node[v] (V) {};
\draw [->] (0,0) -- (1.8,0);
\draw[->,blue] (0,0.7) .. controls (0.4,0.5) .. (V);
\draw[->,blue] (V) .. controls (1.4,0.5) .. (1.8,0.7);
\draw[->,blue] (1.8,1.4) .. controls (0.9,0.2) .. (0,1.4);
\end{tikzpicture}
 ~\right)
 \\
+w_\fso\left(~
\begin{tikzpicture}[scale=0.8]
\draw (0.9,0) node[v] (V) {};
\draw [->] (0,0) -- (1.8,0);
\draw[blue] (0,0.7) .. controls (0.6,0.5) and +(-0.5,-0.7) .. (1.8,1.4);
\draw[<->,blue] (V) .. controls (0.6,0.9) .. (0,1.4);
\draw[->,blue] (V) .. controls (1.2,0.5) .. (1.8,0.7);
\end{tikzpicture}
 ~\right)-w_\fso\left(~
\begin{tikzpicture}[scale=0.8]
\draw (0.9,0) node[v] (V) {};
\draw [->] (0,0) -- (1.8,0);
\draw[<->,blue] (1.8,0.7) .. controls +(-0.6,-0.4) and (0.6,0.9) .. (0,1.4);
\draw[->,blue] (1.8,1.4) .. controls (1.2,0.7) .. (V);
\draw[blue] (V) .. controls (0.6,0.5) .. (0,0.7);
\end{tikzpicture}
 ~\right)\end{multline}
\end{itemize}
The first three graphs on the right are the same as in the defining relation for the $w_\fgl$ weight system. The last two graphs are extended permutation graphs rather than just permutation graphs. Their expression in terms of permutation graphs is determined by the global structure of the original graph and depends on whether the vertices~$r$ and~$r+1$ belong to the same cycle or to two different ones. Assume that these two vertices belong to different cycles in the original graph and the cycle passing through the vertex~$r$ has $a$ vertices different from~$r$, so that the total length of this cycle is $a+1$. Then we have
\begin{equation}\label{eq:rec-so2}
\begin{tikzpicture}[scale=0.8,baseline=0]
\draw (0.9,0) node[v] (V) {};
\draw [->] (0,0) -- (1.8,0);
\draw[blue] (0,0.7) .. controls +(0.6,-0.2) and +(-0.5,-0.7) .. (1.8,1.4);
\draw[<->,blue] (V) .. controls (0.6,0.9) .. (0,1.4);
\draw[->,blue] (V) .. controls (1.2,0.5) .. (1.8,0.7);
\draw[->,blue,dashed] (-0.9,1.3) node {$\scriptstyle a$}  
(0,1.4) .. controls +(-0.6,0.5)  and +(-1.2,0.4) .. (0,0.7);
\end{tikzpicture}
=(-1)^a
\begin{tikzpicture}[scale=0.8,baseline=0]
\draw (0.9,0) node[v] (V) {};
\draw [->] (0,0) -- (1.8,0);
\draw[->,blue] (1.8,1.4) .. controls +(-0.5,-0.7) and +(0.6,-0.2) .. (0,0.7);
\draw[->,blue] (0,1.4) .. controls (0.6,0.9) .. (V);
\draw[->,blue] (V) .. controls (1.2,0.5) .. (1.8,0.7);
\draw[->,blue,dashed] (-0.9,1.3) node {$\scriptstyle a$}  
(0,0.7) .. controls +(-1.2,0.4)  and +(-0.6,0.5) .. (0,1.4);
\end{tikzpicture}
\;,\qquad 
\begin{tikzpicture}[scale=0.8,baseline=0]
\draw (0.9,0) node[v] (V) {};
\draw [->] (0,0) -- (1.8,0);
\draw[<->,blue] (1.8,0.7) .. controls +(-0.6,-0.4) and (0.6,0.9) .. (0,1.4);
\draw[->,blue] (1.8,1.4) .. controls (1.2,0.7) .. (V);
\draw[blue] (V) .. controls (0.6,0.5) .. (0,0.7);
\draw[->,blue,dashed] (-0.9,1.3) node {$\scriptstyle a$}  
(0,1.4) .. controls +(-0.6,0.5)  and +(-1.2,0.4) .. (0,0.7);
\end{tikzpicture}
=(-1)^a
\begin{tikzpicture}[scale=0.8,baseline=0]
\draw (0.9,0) node[v] (V) {};
\draw [->] (0,0) -- (1.8,0);
\draw[<-,blue] (1.8,0.7) .. controls +(-0.6,-0.4) and (0.6,0.9) .. (0,1.4);
\draw[->,blue] (1.8,1.4) .. controls (1.2,0.7) .. (V);
\draw[->,blue] (V) .. controls (0.6,0.5) .. (0,0.7);
\draw[<-,blue,dashed] (-0.9,1.3) node {$\scriptstyle a$}  
(0,1.4) .. controls +(-0.6,0.5)  and +(-1.2,0.4) .. (0,0.7);
\end{tikzpicture}
\;.
\end{equation}
In the case when the vertices $r$ and $r+1$ belong to one cycle, consider the part of this cycle starting at the vertex~$r$ and ending at the vertex~$r+1$, and denote by $a$ the number of vertices on this part excluding the vertices $r$ and $r+1$. Then we have
\begin{equation}\label{eq:rec-so3}
\hskip-1.5em\begin{tikzpicture}[scale=0.8,baseline=0]
\draw (0.9,0) node[v] (V) {};
\draw [->] (0,0) -- (1.8,0);
\draw[blue] (0,0.7) .. controls +(0.6,-0.2) and +(-0.5,-0.7) .. (1.8,1.4);
\draw[<->,blue] (V) .. controls (0.6,0.9) .. (0,1.4);
\draw[->,blue] (V) .. controls (1.2,0.5) .. (1.8,0.7);
\draw[->,blue,dashed] (0.9,2.1) node {$\scriptstyle a$}  
(0,1.4) .. controls +(-2*0.6,2*0.5)  and +(2*0.5,2*0.7) .. (1.8,1.4);
\end{tikzpicture}\hskip-1.5em
=(-1)^a
\hskip-1.5em\begin{tikzpicture}[scale=0.8,baseline=0]
\draw (0.9,0) node[v] (V) {};
\draw [->] (0,0) -- (1.8,0);
\draw[->,blue] (0,0.7) .. controls +(0.6,-0.2) and +(-0.5,-0.7) .. (1.8,1.4);
\draw[<-,blue] (V) .. controls (0.6,0.9) .. (0,1.4);
\draw[->,blue] (V) .. controls (1.2,0.5) .. (1.8,0.7);
\draw[<-,blue,dashed] (0.9,2.1) node {$\scriptstyle a$}  
(0,1.4) .. controls +(-2*0.6,2*0.5)  and +(2*0.5,2*0.7) .. (1.8,1.4);
\end{tikzpicture}\hskip-1.5em
\;,\qquad 
\hskip-1.5em\begin{tikzpicture}[scale=0.8,baseline=0]
\draw (0.9,0) node[v] (V) {};
\draw [->] (0,0) -- (1.8,0);
\draw[<->,blue] (1.8,0.7) .. controls +(-0.6,-0.4) and (0.6,0.9) .. (0,1.4);
\draw[->,blue] (1.8,1.4) .. controls (1.2,0.7) .. (V);
\draw[blue] (V) .. controls (0.6,0.5) .. (0,0.7);
\draw[->,blue,dashed] (0.9,2.1) node {$\scriptstyle a$}  
(0,1.4) .. controls +(-2*0.6,2*0.5)  and +(2*0.5,2*0.7) .. (1.8,1.4);
\end{tikzpicture}\hskip-1.5em
=(-1)^{a+1}
\hskip-1.5em\begin{tikzpicture}[scale=0.8,baseline=0]
\draw (0.9,0) node[v] (V) {};
\draw [->] (0,0) -- (1.8,0);
\draw[<-,blue] (1.8,0.7) .. controls +(-0.6,-0.4) and (0.6,0.9) .. (0,1.4);
\draw[<-,blue] (1.8,1.4) .. controls (1.2,0.7) .. (V);
\draw[<-,blue] (V) .. controls (0.6,0.5) .. (0,0.7);
\draw[<-,blue,dashed] (0.9,2.1) node {$\scriptstyle a$}  
(0,1.4) .. controls +(-2*0.6,2*0.5)  and +(2*0.5,2*0.7) .. (1.8,1.4);
\end{tikzpicture}\hskip-1.5em
\end{equation}
Finally, in the exceptional case when $s(r+1)=r$ we have
\begin{equation}\label{eq:rec-so4}
w_\fso\left(~
\begin{tikzpicture}[scale=0.8]
\draw (0.6,0) node[v] (V1) {} (1.2,0) node[v] (V2) {};
\draw [->] (0,0) -- (1.8,0);
\draw[->,blue] (0,1) .. controls (0.7,0.6) .. (V2);
\draw[->,blue] (V1) .. controls (1.1,0.6) .. (1.8,1);
\draw[->,blue] (V2) .. controls (0.9,0.2) .. (V1);
\end{tikzpicture}
 ~\right)=
 w_\fso\left(~
\begin{tikzpicture}[scale=0.8]
\draw  (0.6,0) node[v] (V1) {} (1.2,0) node[v] (V2) {};
\draw [->] (0,0) -- (1.8,0);
\draw[->,blue] (0,1) .. controls (0.4,0.6) .. (V1);
\draw[->,blue] (V2) .. controls (1.4,0.6) .. (1.8,1);
\draw[->,blue] (V1) .. controls (0.9,0.3) .. (V2);
\end{tikzpicture}
 ~\right)+
(2-C_0)\,w_\fso\left(~
\begin{tikzpicture}[scale=0.8]
\draw (0.9,0) node[v] (V) {};
\draw [->] (0,0) -- (1.8,0);
\draw[->,blue] (V) .. controls (1.2,0.5) .. (1.8,1);
\draw[->,blue] (0,1) .. controls (0.6,0.5) .. (V);
\end{tikzpicture}
 ~\right).
\end{equation}
\end{definition}

\begin{theorem}\label{th:so}
1. The defining relations determine the invariant $w_\fso$ uniquely.

2. For chord diagrams (corresponding to involutions without fixed points) the invariant~$w_\fso$ is a weight system, i.e. it satisfies the $4$-term relation.
\end{theorem}

The proof of Theorem~\ref{th:so} goes as follows. Using defining relations any graph can be reduced to a product of independent cycles of both even and odd length. Next we observe that any cycle is conjugate to its inverse in the permutation group. It follows that the difference between a cycle and its inverse can be expressed modulo the relations in terms of permutations on smaller number of vertices. If the length of the cycle is even this relation says nothing about the value of $w_\fso$ on this cycle. But in the case when the length of the cycle is odd we obtain that this cycle taken twice is expressed in terms of permutations with smaller number of elements. Applying this procedure repeatedly we finally express any permutation (permutation graph) as a polynomial (a linear combination of monomials) in the standard cycles of even length~$2k$ for which the value of $w_\fso$ is denoted by~$C_{2k}$.

Next we need to show that the obtained polynomial expression for a given permutation in terms of generators is unique, that is, independent of the order in which we applied the recurrence relations to obtain this expression.

For that we consider the weight systems associated with the classical Lie algebra $\fso(N)$, and define the extension of the corresponding weight system $w_{\fso(N)}$ to the set of permutations using~\eqref{eq:wsonsigma}. Then, we prove in the next section that this invariant satisfies defining relations of~$w_\fso$ weight system, see Theorem~\ref{th:soNspN}.

We conclude the proof of Theorem~\ref{th:so} by the arguments similar to those for the case of the $w_\fgl$ weight system. Using recurrence relations, we can express the value of $w_{\fso(N)}$ on a given permutation as a polynomial in the Casimir elements corresponding to cyclic permutations of even length. The Casimir elements belong to the center of $U\fso(N)$ which implies that $w_{\fso(N)}(s)$ also belongs to the center.
Moreover, the Casimir elements $C_{\fso(N),2},C_{\fso(N),4},\dots,C_{\fso(N),[N/2]}$ are algebraically independent. It follows that for a fixed permutation and sufficiently large~$N$ the polynomial expression for this permutation in terms of Casimirs is uniquely defined (for the specialization $C_0=N$). The coefficients of this polynomial depend on~$N$ in a polynomial way and these polynomials in~$N$ are uniquely reconstructed from a finite collection of (sufficiently large) values of~$N$.

The multiplicativity of $w_\fso$ holds because it obviously holds for $w_{\fso(N)}$. The $4$-term relation also follows from the fact that it holds in the case of the $\fso(N)$ weight system (as for a weight system associated with a metrized Lie algebra). This completes the proof of Theorem~\ref{th:so}.

\begin{remark}\label{rem:Casimirs}
1. If $N$ is odd, the center of $U\fso(N)$ is generated by Casimirs. 
If $N=2N_1$ is even, then the center of $U\fso(2N_1)$ along with the Casimir elements contains also an extra generator which is not expressed in terms of the Casimirs, see\cite{nwachuku1979}. The arguments of the proof above imply that the $\fso(2M)$ weight system takes actually values in the subring of $ZU\fso(2M)$ generated by Casimirs.

2. As it is mentioned above, for odd~$m$ the element $C_m$ defined as the value of $w_\fso$ on the standard cycle of length~$M$ can be expressed as a polynomial in the generators $C_{m'}$ with even indices~$m'$. For example, applying directly defining relations we find for small odd~$m$,
\begin{align*}
C_1&=0,\\
C_3&=\frac{C_0 C_2}{2}-C_2,\\
C_5&=-\frac{1}{4} C_2 C_0^3+\frac{5}{4} C_2 C_0^2-2 C_2 C_0+\frac{3 C_4 C_0}{2}+C_2-2 C_4-\frac{C_2^2}{2},
\end{align*}
and so on. In fact, algebraic relations permitting one to find expressions for $C_m$ with odd~$m$ in a closed form can be written as follows. Let us combine $C_m$ into the following generating series
$$
F(u)=1-\frac{u-\frac{C_0-1}{2}}{u-\frac{C_0-2}{2}}\;\sum_{m=0}^\infty C_m u^{-m-1}
$$
in the inverse powers of an independent variable~$u$. (The reason why the generating series is written in such fancy way is because it appears in this form in the Perelomov-Popov formula for the Lie algebras $\fso(N)$ and $\fsp(N)$, see~Appendix).
Then the defining relation for the elements $C_m$  with odd~$m$ reads
$$
F(u)\,F(C_0-1-u)=1.
$$
For the proof of this relation see~Appendix.
\end{remark}

\subsection{Proof of cyclic invariance of~$w_\fgl$ and ~$w_\fso$}
\begin{theorem}\label{thm:cyclic}
    The values of the weight systems $w_{\fgl}$ and $w_{\fso}$ depend on the cyclic order on the set of permuted elements only. In other words, a conjugation of a permutation $s\in S(m)$ by the cyclic permutation $1\mapsto 2\mapsto\dots\mapsto m\mapsto 1$ preserves the values of both weight systems.
\end{theorem}
\begin{proof}
By induction on the number of permuted elements, let $s\in S(m)$ and $s^{cyc}\in S(m)$ be the conjugation of $s$ by the cyclic permutation $1\mapsto 2\mapsto\dots\mapsto m\mapsto 1$. Assume the cyclic invariance of $w_\fgl$ and $w_\fso$ on all permutations with less than $m$ elements. Since the right side of the recursion of $w_\fgl$ and $w_\fso$ is the difference of switching two neighbouring legs, the left side contains only the weight systems on permutations with less elements. By induction, \[
w_\fg(s)-w_\fg(s^{cyc})=w_\fg(s_{i,i+1})-w_\fg(s_{i,i+1}^{cyc}), 1<i<m-1,\fg=\fgl,\fso.
\]
The difference is just the difference after switching two neighbouring legs. You can always reorder the elements as the product of standard cyclic permutations. By the multiplicative commutativity of Casimirs and the natural cyclic invariance of standard cyclic permutations, the difference $w_\fg(s)-w_\fg(s^{cyc})=0$.
\end{proof}

\section{Recursion for the weight systems $w_{\fso(N)}$, $w_{\fsp(2M)}$, and $w_{\fosp(N|2M)}$}

In this section, we define extensions of the weight systems $w_{\fso(N)}$, $w_{\fsp(2M)}$, and $w_{\fosp(N|2M)}$ associated with the Lie algebras and Lie superalgebras to the set of permutations and prove that these extensions do satisfy the relations of Definition~\ref{def:so}. Namely, we prove the following.

\begin{theorem}\label{th:soNspN}
The extension to the set of permutations for the weight systems associated with the Lie algebras $\fso(N)$ and $\fsp(2M)$, and Lie superalgebra $\fosp(N|2M)$ take values in the center of the corresponding universal enveloping algebras and obey the defining relations of the universal $w_\fso$ weight system, with the specialization of $C_k$ to the standard Casimir elements for $k\ge1$, and with $C_0=N$ for the Lie algebra $\fso(N)$, $C_0=-2M$ for the Lie algebra $\fsp(2M)$, and $C_0=N-2M$ for the Lie superalgebra~$\fosp(N|2M)$.
\end{theorem}

\begin{remark}
    The Lie algebras $\fso(N)$ and $\fsp(2M)$ are special cases of the Lie superalgebra $\fosp(N|2M)$, namely, $\fso(N)=\fsp(N|0)$ and $\fsp(2M)=\fosp(0|2M)$. By that reason, the statement of Theorem concerning the Lie algebras  $\fso(N)$ and $\fsp(2M)$ follows from the one for the Lie superalgebra $\fosp(N|2M)$. However, we consider the cases of the Lie algebras $\fso(N)$ and $\fsp(2M)$ independently because the arguments admit some simplifications in these cases.
\end{remark}

\subsection{$w_{\fso(N)}$}

The Lie algebra $\fso(N)$ is the Lie subalgebra of $\fgl(N)$ spanned by $X_{ij}=E_{ij}-E_{\bar j\bar i}$, $i,j=1,\dots,N$, where
$$
\bar i=N+1-i.
$$
These generators satisfy linear relations
$$
X_{i,j}+X_{\bar j\bar i}=0
$$
and the multiplication in the universal enveloping algebra $U\fso(N)$ satisfies the quadratic relations
$$
X_{ij}X_{kl}-X_{kl}X_{ij}
=\delta_{kj}X_{il}-\delta_{il}X_{kj}
-\delta_{\bar l j}X_{i\bar k}+\delta_{i\bar k}X_{\bar lj}.
$$
For any permutation $s\in S(l)$ we define
$$w_{\fso(N)}(s)=\sum_{i_1,\dots,i_l=1}^{N}X_{i_1i_{s(1)}}\dots X_{i_li_{s(l)}}\in U\fso(N).$$
The Casimir element~$C_m$ corresponds to the standard cycle $1\mapsto 2\mapsto\dots\mapsto m\mapsto 1$
$$C_m=\sum_{i_1,\dots,i_l=1}^{N}X_{i_1i_{2}}\dots X_{i_mi_{1}}\in U\fso(N).$$

The sum in the definition of~$w_{\fso(N)}(s)$ can be represented graphically as follows. We take the digraph of~$s$ and assign an additional marking to each its vertex in the range from $1$ to $N$. Then, to the edge connecting vertices marked with $i$ and $j$ we associate the generator $X_{ij}$, multiply these generators for all edges in the order corresponding to their ordering in the graph, and sum up over all possible ways to put these additional markings on vertices.

The recursion relations~\eqref{eq:rec-so1}--\eqref{eq:rec-so4} just imitate the commutation relations in the Lie algebra $\fso(N)$.
We have the Quadratic relations in $\fso(N)$:\begin{align*}&X_{i_ki_{s(k)}}X_{i_{k+1}i_{s(k+1)}}-X_{i_{k+1}i_{s(k+1)}}X_{i_ki_{s(k)}}\\
=&\delta_{i_{k+1}i_{s(k)}}X_{i_ki_{s(k+1)}}-\delta_{i_ki_{s(k+1)}}X_{i_{k+1}i_{s(k)}}
+\delta_{i_k\bar i_{k+1}}X_{\bar i_{s(k+1)}i_{s(k)}}-\delta_{\bar i_{s(k+1)} i_{s(k)}}X_{i_k\bar i_{k+1}}.
\end{align*}
The left side as well as the first two terms in the right side naturally correspond to the terms in the commutation relations in the $\fgl(N)$ case  and they are represented by the same diagrams as in~\eqref{eq:rec-gl} and~\eqref{eq:rec-gl2}, see \cite{Y1}.

Now we consider the last two terms in the above commutation relations. Both
$\delta_{i_k\bar i_{k+1}}$ and $X_{i_k\bar i_{k+1}}$ are not terms in the standard  definition of $w_\fso$. Since we have $X_{ij}=-X_{\bar j\bar i}$, we change the next term whose index contains $i_{k+1}$:$X_{i_{s^{-1}(k+1)} i_{k+1}}$ as $-X_{\bar i_{k+1}\bar i_{s^{-1}(k+1)} }$; and change the next term $X_{i_{s^{-2}(k+1)} i_{s^{-1}(k+1)}}$ as $-X_{\bar i_{s^{-1}(k+1)} \bar i_{s^{-2}(k+1)}}$; and so on until we run back to the terms $X_{\bar i_{s(k+1)}i_{s(k)}}$ and $\delta_{\bar i_{s(k+1)} i_{s(k)}}$ in the  $w_\fso(s)$. Since $w_{\fso(N)}$ takes the sum over all indices, we can ignore the bar on the indices we changed. Switching the indices means we change the direction on the edge in the permutation diagram.  The last two terms also correspond to the permutation diagrams~\eqref{eq:rec-so2} and~\eqref{eq:rec-so3} in the last two terms of the recursion relations~\eqref{eq:rec-so1}. 

The last relation~\eqref{eq:rec-so4} can be regarded as a special case of~\eqref{eq:rec-so1} if we treat a loop without vertices in the diagrams on the right hand side as a factor~$C_0$.
This completes the proof of Theorem~\ref{th:soNspN} for the case of the Lie algebra~$\fso(N)$.

\subsection{$w_{\fsp(2M)}$}
The Lie algebra $\fsp(2M)$ is the Lie subalgebra of $\fgl(2M)$ spanned by $X_{ij}=E_{ij}-\epsilon_i\epsilon_jE_{\bar j\bar i}$, $i,j=1,\dots,N$, where
$$
\bar i=2M+1-i,\qquad
\epsilon_i=\begin{cases}
 1,&i=1,\dots,M,\\
 -1,&i=M+1,\dots,2M.
\end{cases}
$$
These generators satisfy linear relations
$$
X_{i,j}+\epsilon_i\epsilon_jX_{\bar j\bar i}=0
$$
and the multiplication in the universal enveloping algebra $U\fsp(2M)$ is subject to the quadratic relations
$$
X_{ij}X_{kl}-X_{kl}X_{ij}
=\delta_{kj}X_{il}-\delta_{il}X_{kj}
-\epsilon_i\epsilon_j(\delta_{\bar l j}X_{i\bar k}-\delta_{i\bar k}X_{\bar lj}).
$$
For any permutation $s\in S(m)$ consisting of $r$ cycles, we define
$$w_{\fsp(2M)}(s)=(-1)^{m+r}\sum_{i_1,\dots,i_m=1}^{N}X_{i_1i_{s(1)}}\dots X_{i_mi_{s(m)}}\in U\fsp(2M).$$
An extra sign appearing in front of the sum is a matter of normalization. We put it in order to obtain the weight system $w_{\fsp(2M)}$ as a specialization of the universal weight system $w_{\fso}$, see the arguments below. The same sign convention is used also in the definition of the Casimir elements,
$$C_m=(-1)^{m+1}\sum_{i_1,\dots,i_m=1}^{N}X_{i_1i_{2}}\dots X_{i_mi_{1}}\in U\fsp(2M).$$

The commutation relations in the Lie algebra~$\fsp(2M)$ involve extra sign factors $\epsilon_i$ so that interpretation of the relation~\eqref{eq:rec-so1} similar to $\fso(N)$ case does not hold. however, it turns out that if we use directly relations~\eqref{eq:rec-so2} and~\eqref{eq:rec-so3} instead of involving permutation graphs only, then all these $\epsilon$-factors cancel out and the commutation relations admit graphical presentations. In particular, if we define the $w_{\fsp(2M)}$ invariant for a permutation by a formula similar to the $\fso$-case, without $(-1)^{m+r}$ factor, it would also admit a recursion in terms of permutation graphs similar to~\eqref{eq:rec-so2}--\eqref{eq:rec-so3}. A careful account of the contribution of the $\epsilon$-factors shows that the obtained relations will differ from~\eqref{eq:rec-so2}--\eqref{eq:rec-so3} just by the change of the signs in the contribution of the last two graphs on the right hand side of~\eqref{eq:rec-so2}. These two graphs are distinguished exactly by the parity of the sum of length of all the cycles increased by one, that is, the parity of $m+r$. It follows that with an additional factor $(-1)^{r+m}$ involved in the definition of $w_{\fsp(2M)}$ the commutation relations for this invariant are exactly the same as for the invariant~$w_{\fso}$, that is,~\eqref{eq:rec-so2}~\eqref{eq:rec-so3} and~\eqref{eq:rec-so4} with $C_0=-2M$.

Note that the insertion of the factor $(-1)^{r+l}$ is not essential for the study of $\fsp(2M)$ weight system on chord diagrams: it corresponds to just a rescaling of the scalar product.


\subsection{$w_{\fosp(N|2M)}$}

Now let us turn to the general cade.
Consider the $\Z_2$-graded coordinate space $\C^{N+2M}=\C^N\oplus\C^{2M}$ with the following convention on the set of indices $\{1,2,\dots,N+2M\}$ labelling the coordinates in these subspaces: we say that the index $i$ is even if $M+1\le i \le M+N$ and $i$ is odd otherwise. Following this convention we set
\begin{align*}
\bar i&=N+2M+1-i,\\
\eta_i&=\begin{cases} 0,& M+1\le i \le M+N,\\
1,& 1\le i \le M\text{ or }N+M+1\le i \le N+2M,\end{cases}\quad
\sigma_i=(-1)^{\eta_i},\\
\epsilon_i&=\begin{cases} 1,& 1\le i\le M+N,\\
-1,& N+M+1\le i \le N+2M.\end{cases}
\end{align*}
The Lie superalgebra $\fosp(N|2M)$ is a subalgebra of $\fgl(N|2M)$ spanned by the generators 
$$
X_{ij}=E_{ij}-(-1)^{\eta_i\eta_j}\epsilon_{\bar i}\epsilon_j E_{\bar j\bar i}, i,j=1,\dots,N+2M.
$$
These generators satisfy the linear relations
$$
X_{ij}+(-1)^{\eta_i\eta_j}\epsilon_{\bar i}\epsilon_j X_{\bar j\bar i}=0.
$$
and the multiplication in the universal enveloping algebra $U\fosp(N|2M)$ is subject to the quadratic relations
\begin{multline*}
X_{ij}X_{kl}-(-1)^{(\eta_i+\eta_j)(\eta_k+\eta_l)}X_{kl}X_{ij}\\
=\delta_{kj}X_{il}-(-1)^{(\eta_i+\eta_j)(\eta_k+\eta_l)}\delta_{il}X_{kj}
-(-1)^{\eta_i\eta_j}\epsilon_{\bar i}\epsilon_j(\delta_{\bar l j}X_{i\bar k}-(-1)^{(\eta_i+\eta_j)(\eta_k+\eta_l)}\delta_{i\bar k}X_{\bar lj}).
\end{multline*}
The Casimir elements are $C_0=N-2M$ and
$$
C_k=\sum_{i_1,\dots,i_k=1}^N(-1)^{\eta_{i_2}+\dots+\eta_{i_k}} X_{i_1i_2}X_{i_2i_3}\dots X_{i_ki_1}
\quad (k\ge1).
$$


Let $s\in S(m)$ be a permutation. We say that an index $i\in\{1,\dots,l\}$ is \emph{distinguished} if $s(i)>1$. We say that a pair of indices $(i,j)$, $i<j$, is distinguished if the pairs of real numbers $(i+1/3,s(i)-1/3)$ and $(j+1/3,s(j)-1/3)$ alternate on the real line. Denote the sets of distinguished indices and pairs of indices by $P_1$ and $P_2$, respectively, and set
$$
f_{s}(\tau_1,\dots,\tau_m)=\sum_{i\in P_1}\tau_i+\sum_{\{i,j\}\in P_2}\tau_i\tau_j.
$$
By construction $f_s$ is a polynomial in $\tau_1,\dots,\tau_m$.

The extension of the $w_{\fosp(N|2M)}$ weight system to the set of permutations is defined by
$$
w_{\fosp(N|2M)}(s)=\sum_{i_1,\dots,i_m=1}^{N+2M}(-1)^{f_s(\eta_{i_1},\dots,\eta_{i_m})}X_{i_1i_{s(1)}}\dots X_{i_mi_{s(m)}}\in U\fosp(N|2M).
$$
The definition of the sign function $(-1)^{f_s}$ implies that the weight systems associated with the Lie algebras $\fso(N)=\fosp(N|0)$ and $\fsp(2M)=\fosp(0,2M)$ are special cases of more general Lie superalgebra $\fosp(N|2M)$.

The verification of the  fact that the invariant $w_{\fosp(N|2M)}$ satisfies the same recursion relations~\eqref{eq:rec-so1}--\eqref{eq:rec-so4} of the universal $w_{\fso}$ weight system with $C_0=N-2M$ is straightforward. Useful lemmas simplifying this verification are formulated in~~\cite{Y2}.

For the Lie superalgebra version, we need to take care of the sign function. The left side and the first two terms in the right side are naturally correspond to the permutation diagrams in  the recursion relations~\eqref{eq:rec-so1}, see \cite{Y2}. 

Now we look at the last two terms on the right. We have the Quadratic relations:
\begin{multline*}
X_{ij}X_{kl}-(-1)^{(\eta_i+\eta_j)(\eta_k+\eta_l)}X_{kl}X_{ij}\\
=\delta_{kj}X_{il}-(-1)^{(\eta_i+\eta_j)(\eta_k+\eta_l)}\delta_{il}X_{kj}
-(-1)^{\eta_k\eta_l}\epsilon_{\bar k}\epsilon_l(\delta_{\bar l j}X_{i\bar k}-(-1)^{(\eta_i+\eta_j)(\eta_k+\eta_l)}\delta_{i\bar k}X_{\bar lj}).
\end{multline*}

Draw the graph as follows (the position of $i,j,k,l$ is just an example, and in the other situations, the calculations are similar):
\begin{align*}
  \begin{tikzpicture}[scale=0.8]
    \draw (0,1.4) node[left] {\tiny $j$};
    \draw (0,.7) node[left] {\tiny $i$};
    \draw (1.8,1.4) node[right] {\tiny $k$};
    \draw (1.8,.7) node[right] {\tiny $l$};
\draw  (0.7,0) node[v] (V1) {} (1.1,0) node[v] (V2) {};
\draw [->] (0,0) -- (1.8,0);
\draw[->,blue] (0,0.7) .. controls (0.4,0.4) .. (V1);
\draw[->,blue] (V1) .. controls (0.4,0.9) .. (0,1.4);
\draw[->,blue] (1.8,1.4) .. controls (1.4,0.9) .. (V2);
\draw[->,blue] (V2) .. controls (1.4,0.4) .. (1.8,0.7);
    \draw (0,-0.7) node { $f_s X_{ij}X_{kl}$};
  \end{tikzpicture}-
\begin{tikzpicture}[scale=0.8]
    \draw (0,1.4) node[left] {\tiny $j$};
    \draw (0,.7) node[left] {\tiny $i$};
    \draw (1.8,1.4) node[right] {\tiny $k$};
    \draw (1.8,.7) node[right] {\tiny $l$};
\draw (0.6,0) node[v] (V1) {} (1.2,0) node[v] (V2) {};
\draw [->] (0,0) -- (1.8,0);
\draw[->,blue] (0,0.7) .. controls (0.6,0.5) .. (V2);
\draw[->,blue] (V2) .. controls (0.6,0.9) .. (0,1.4);
\draw[->,blue] (1.8,1.4) .. controls (1.2,0.9) .. (V1);
\draw[->,blue] (V1) .. controls (1.2,0.5) .. (1.8,0.7);
    \draw (0,-0.7) node { $(f_s+(-1)^{(\eta_i+\eta_j)(\eta_k+\eta_l)} )X_{kl}X_{ij}$};
  \end{tikzpicture}=\\
\begin{tikzpicture}[scale=0.8]
    \draw (0,1.4) node[left] {\tiny $j$};
    \draw (0,.7) node[left] {\tiny $i$};
    \draw (1.8,1.4) node[right] {\tiny $k$};
    \draw (1.8,.7) node[right] {\tiny $l$};
\draw (0.9,0) node[v] (V) {};
\draw [->] (0,0) -- (1.8,0);
\draw[->,blue] (0,0.7) .. controls (0.4,0.5) .. (V);
\draw[->,blue] (V) .. controls (1.4,0.5) .. (1.8,0.7);
\draw[->,blue] (1.8,1.4) .. controls (0.9,0.2) .. (0,1.4);
    \draw (0,-0.7) node { $f_{s|\eta_l=\eta_k} \delta_{kj}X_{il}$};
    \end{tikzpicture}-
\begin{tikzpicture}[scale=0.8]
    \draw (0,1.4) node[left] {\tiny $j$};
    \draw (0,.7) node[left] {\tiny $i$};
    \draw (1.8,1.4) node[right] {\tiny $k$};
    \draw (1.8,.7) node[right] {\tiny $l$};
\draw (0.9,0) node[v] (V) {};
\draw [->] (0,0) -- (1.8,0);
\draw[->,blue] (0,0.7) .. controls +(0.5,-0.5) and +(-0.5,-0.5) .. (1.8,0.7);
\draw[->,blue] (1.8,1.4) .. controls +(-0.6,-0.7) .. (V);
\draw[->,blue] (V) .. controls +(-0.3,0.7) .. (0,1.4);
    \draw (0,-0.7) node { $(f_{s|\eta_l=\eta_i}+(-1)^{\eta_i\eta_j+\eta_k\eta_i+\eta_j\eta_k} )\delta_{il}X_{kj}$};
  \end{tikzpicture}-\\
  - \Biggl( 
\begin{tikzpicture}[scale=0.8]
    \draw (0,1.4) node[left] {\tiny $j$};
    \draw (0,.7) node[left] {\tiny $i$};
    \draw (1.8,1.4) node[right] {\tiny $k$};
    \draw (1.8,.7) node[right] {\tiny $l$};
\draw (0.9,0) node[v] (V) {};
\draw [->] (0,0) -- (1.8,0);
\draw[->,red] (1.8,0.7) .. controls +(-0.6,-0.4) and (0.6,0.9) .. (0,1.4);
\draw[<-,red] (1.8,1.4) .. controls (1.2,0.7) .. (V);
\draw[<-,blue] (V) .. controls (0.6,0.5) .. (0,0.7);
    \draw (0,-0.7) node { $(f_{s|\eta_l=\eta_j}+(-1)^{\eta_k\eta_l} )\delta_{\bar l j}X_{i\bar k}$};
  \end{tikzpicture}-
\begin{tikzpicture}[scale=0.8]
    \draw (0,1.4) node[left] {\tiny $j$};
    \draw (0,.7) node[left] {\tiny $i$};
    \draw (1.8,1.4) node[right] {\tiny $k$};
    \draw (1.8,.7) node[right] {\tiny $l$};
\draw (0.9,0) node[v] (V) {};
\draw [->] (0,0) -- (1.8,0);
\draw[->,red] (0,0.7) .. controls (0.6,0.5) and +(-0.5,-0.7) .. (1.8,1.4);
\draw[->,blue] (V) .. controls (0.6,0.9) .. (0,1.4);
\draw[<-,red] (V) .. controls (1.2,0.5) .. (1.8,0.7);
    \draw (0,-0.7) node { $(f_{s|\eta_k=\eta_i}+(-1)^{\eta_i\eta_j+\eta_j\eta_l} )\delta_{i\bar k}X_{\bar lj}$};
  \end{tikzpicture}
  \Biggr)
\end{align*}\\

The difference of the linear terms is given by the sum of the terms $\epsilon_k\epsilon_{\bar k}$ appearing during the switch of the direction of edges.

The difference of the quadratic terms depends on crossing only. When you switch the direction of an edge which means changing $X_{ij}$ to $X_{\bar j\bar i}$,  the only difference of quadratic terms in the sign function is $\eta_i\eta_j$. It just matches the linear relation $X_{ij}=-(-1)^{\eta_i\eta_j}\epsilon_{\bar i}\epsilon_j X_{\bar j\bar i}$. \qed

\section{Computations}
Let $A_n$ be the space of chord diagrams with $n$ chords modulo $4$-term relations. Then the weight systems $w_\fgl$ and $\fso$ can be thought of as linear functions on $A_n$. The table below gives some results from computer experiments showing how strong these weight systems are.
$$\begin{array}{c|ccccccccccccccccc}
N&1&2&3&4&5&6&7\\\hline
\dim A(N)&1&2&3&6&10&19&33\\
\dim \bigl(\mathop{\rm ker} w_\fgl|_{A(N)}\bigr)&0&0&0&0&0&1&4\\
\dim \bigl(\mathop{\rm ker} w_\fgl|_{A(N)}\cap \mathop{\rm ker} w_\fso|_{A(N)}\bigr) &0&0&0&0&0&1&3
\end{array}
$$
This table shows that up to degree $5$ the weight system $w_{\fgl}$ contains all possible weight systems. In degree~$6$ there is a unique linear combination of chord diagrams, up to a factor, which is nontrivial modulo $4$-term relations but $w_{\fgl}$ vanishes on it. A straightforward verification shows that $w_{\fso}$ vanishes on it as well. In degree $7$, there are $4$ linearly independent linear combinations of chord diagrams in the kernel of~$w_{\fgl}$. The weight system $w_{\fso}$ is not trivial on this kernel. Therefore, there is an element $h\in A_7$ such that $w_{\fgl}(h)=0$  but $w_\fso(h)\ne 0$. This shows that the universal weight system $w_\fso$  while being much weaker than the $w_{\fgl}$ weight system is however independent and cannot be induced from $w_\fgl$. 

In order to represent this computation more explicitly, we recall that $A=\bigoplus A_k$ is a polynomial algebra with the number of generators of degrees $1,2,3,\dots$ equal to $1,1,1,2,3,5,8,\dots$, see~\cite{Chmutov2012}. One of the possible choices for the generators in degrees $\le7$ is as follows
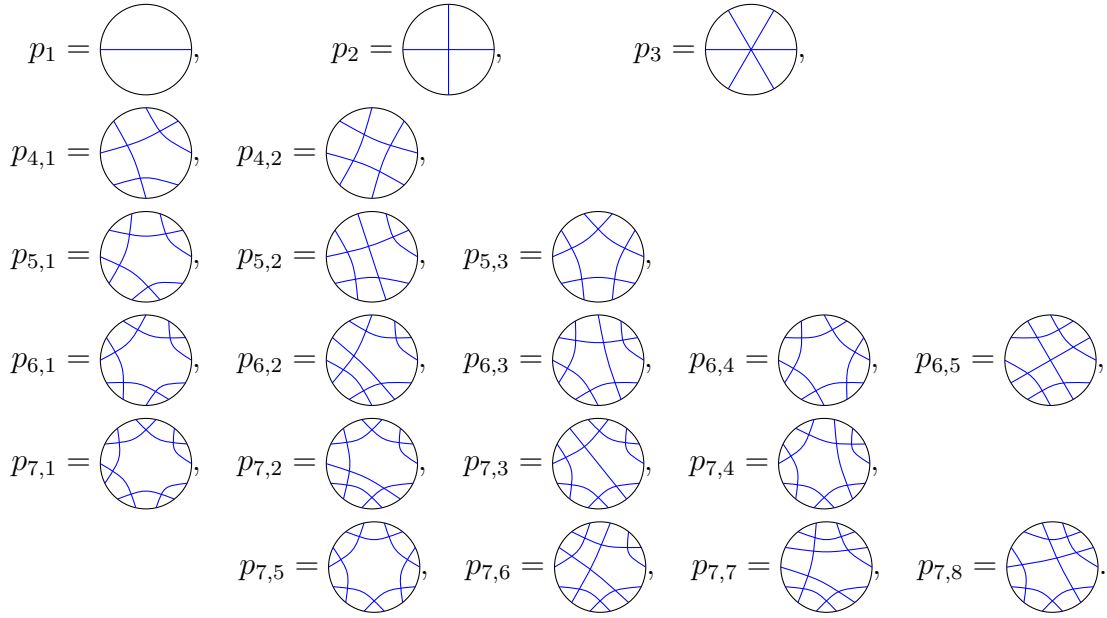
\def\cd#1#2{\def\factor{180/#1}\tikz[scale=0.6]{\draw (0,0) circle [radius=1]; #2}}
\def\chord(#1,#2);{
\coordinate (->) at ($(0,0)!1!\factor*(#1-1):(1,0)$);
\coordinate (b) at ($(0,0)!1!\factor*(#2-1):(1,0)$);
\draw[blue] (->) .. controls ($0.35*(->) + 0.35*(b)$) .. (b);}
\begin{align*}
p_{1}&=\cd{1}{\chord(1,2);},
\qquad\qquad p_{2}=\cd{2}{\chord(1,3);\chord(2,4);},
\qquad\qquad p_{3}=\cd{3}{\chord(1,4);\chord(2,5);\chord(3,6);},
\\ p_{4,1}&=\cd{4}{\chord(1,3);\chord(2,5);\chord(4,7);\chord(6,8);},
\quad p_{4,2}=\cd{4}{\chord(1,4);\chord(2,7);\chord(3,6);\chord(5,8);},
\\p_{5,1}&=\cd{5}{\chord(1,3);\chord(2,5);\chord(4,7);\chord(6,9);\chord(8,10);},
\quad p_{5,2}=\cd{5}{\chord(1,3);\chord(2,6);\chord(4,9);\chord(5,8);\chord(7,10);},
\quad p_{5,3}=\cd{5}{\chord(1,4);\chord(2,9);\chord(3,6);\chord(5,8);\chord(7,10);},
\\p_{6,1}&=\cd{6}{\chord(1,3);\chord(2,5);\chord(4,7);\chord(6,9);\chord(8,11);\chord(10,12);},
\quad p_{6,2}=\cd{6}{\chord(1,3);\chord(2,5);\chord(4,8);\chord(6,11);\chord(7,10);\chord(9,12);},
\quad p_{6,3}=\cd{6}{\chord(1,3);\chord(2,6);\chord(4,11);\chord(5,8);\chord(7,10);\chord(9,12);},
\quad p_{6,4}=\cd{6}{\chord(1,4);\chord(2,11);\chord(3,6);\chord(5,8);\chord(7,10);\chord(9,12);},
\quad p_{6,5}=\cd{6}{\chord(1,4);\chord(2,8);\chord(3,6);\chord(5,11);\chord(7,10);\chord(9,12);},
\\p_{7,1}&=\cd{7}{\chord(1,3);\chord(2,5);\chord(4,7);\chord(6,9);\chord(8,11);\chord(10,13);\chord(12,14);},
\quad p_{7,2}=\cd{7}{\chord(1,3);\chord(2,5);\chord(4,7);\chord(6,10);\chord(8,13);\chord(9,12);\chord(11,14);},
\quad p_{7,3}=\cd{7}{\chord(1,3);\chord(2,5);\chord(4,8);\chord(6,13);\chord(7,10);\chord(9,12);\chord(11,14);},
\quad p_{7,4}=\cd{7}{\chord(1,3);\chord(2,6);\chord(4,13);\chord(5,8);\chord(7,10);\chord(9,12);\chord(11,14);},
\\&\hskip5.8em p_{7,5}=\cd{7}{\chord(1,4);\chord(2,13);\chord(3,6);\chord(5,8);\chord(7,10);\chord(9,12);\chord(11,14);},
\quad p_{7,6}=\cd{7}{\chord(1,3);\chord(2,6);\chord(4,10);\chord(5,8);\chord(7,13);\chord(9,12);\chord(11,14);},
\quad p_{7,7}=\cd{7}{\chord(1,4);\chord(2,7);\chord(3,6);\chord(5,10);\chord(8,13);\chord(9,12);\chord(11,14);},
\quad p_{7,8}=\cd{7}{\chord(1,4);\chord(2,8);\chord(3,6);\chord(5,13);\chord(7,10);\chord(9,12);\chord(11,14);}.
\end{align*}

Applying straightforwardly the recurrence relations and using computer we find explicitly the values of the $\fgl$ and $\fso$ weight systems on these generators, for example,
\begin{align*}
w_{\fgl}(p_1)&=C_2,&w_{\fgl}(p_2)&=C_1^2+C_2^2-C_0 C_2,
\\w_{\fso}(p_1)&=C_2
&w_{\fso}(p_2)&=C_2^2-2\,C_0 C_2+4 C_2
\\[2ex]
w_{\fgl}(p_3)&=C_2^3 -2\,C_0 C_2^2 +C_0^2 C_2 +2\,C_1^2 C_2 -C_0 C_1^2, \hskip-10em
\\ w_{\fso}(p_3)&=C_2^3 -4\,C_0 C_2^2 +8\,C_2^2 +4\,C_0^2 C_2 -16\,C_0 C_2 +16\,C_2, \hskip-10em
\end{align*}
and so on. Define
\begin{multline*}
h=3\,p_{1}^5p_{2}-15\,p_{1}^3p_{2}^2 -87\,p_{1}p_{2}^3 -6\,p_{1}^4p_{3}  +255\,p_{1}^2p_{2}p_{3}-12\,p_{2}^2 p_{3} -171\,p_{1}p_{3}^2
\\-99\,p_{1}^3p_{4,1} +133\,p_{3}p_{4,1} +3\,p_{1}^3p_{4,2} -28\,p_{3}p_{4,2} +21\,p_{1}p_{2}p_{4,2}
+45\,p_{1}^2p_{5,1} -24\,p_{1}^2p_{5,2}
\\+18\,p_{1}^2p_{5,3} -12\,p_{2}p_{5,3} +17\,p_{1}p_{6,1} +\,p_{1}p_{6,2}
-30\,p_{1}p_{6,3} -8\,p_{1}p_{6,4} +20\,p_{1}p_{6,5}
\\-46\,p_{7,1} +11\,p_{7,2} +18\,p_{7,3} +8\,p_{7,4} -6\,p_{7,5} -20\,p_{7,6} +5\,p_{7,7} +6\,p_{7,8}\in A_7
\end{multline*}
Then, substituting the computed values of the corresponding weight systems we get 
\begin{align*}
w_{\fgl}(h)&=0,\quad\text{while}\\
w_{\fso}(h)&=192\; C_0\; (C_0-6)\,  (-24\,C_2 +20\,C_0C_2 -4\,C_0^2C_2  
\\&\hskip 15 em
-2\,C_2^2 -4\,C_0C_2^2 +C_0^2C_2^2 +16\,C_4 -2\,C_0C_4).
\end{align*}

\section*{Concluding remarks and perspectives}
\begin{itemize}
\item  So far, we give all the Lie algebra weight system correspond for all the infinite Complex simple Lie algebra series, which are all the classical Lie algebras $\fsl(N)(\fgl(N))$, $\fso(N)$ and $\fsp(2M)$. Also we give all the Lie superalgebra weight system correspond for all the infinite Complex simple Lie superalgebra series $\fsl(N|M)(\fgl(N|M))$ and $w\fosp(N|2M)$ except the strange Lie superalgebra series $\fp(N)$ and $\fq(N)$. The construction in this paper can be also extend to $\fp(N)$ and $\fq(N)$ by some generalization. Lie superalgebra $\fp(N)$ has trivial center of the universal enveloping algebra $U(\fp(N))$, so the correspond the $\fp(N)$ weight system on permutations is also trivial. The $\fq(N)$ weight system on permutations is not trivial, but it takes value $0$ on all chord diagrams, and even on all permutations having at least one cycle of even length.  The paper \cite{KY} is in preparation.

\item  What is the combinatorial meaning of the generalized Vassiliev relation? Are there natural context in the knot theory or its variations where these kind of relations appear naturally?

\item  At the moment, we have three independent examples of polynomial invariant of permutations satisfying the generalised 4-term relations: $w_{\fgl}$, $w_{\fso}$, and $w_{\fq}$. Are there other examples?
\end{itemize}

\appendix
\section{Perelomov-Popov formula for Casimirs}

The Perelomov-Popov formula~\cite{Perelomov-Popov} (and its generalizations for the cases of Lie superalgebras~\cite{NR}) expresses Casimir elements as explicitly given elements in the ring of (super)symmetric functions. In particular, it allows one to express Casimir elements in terms of those of them which are algebraically independent. In this Appendix we review this formula for the case of~$\fosp(N|2M)$ Lie superalgebra following~\cite{NR}. We present it in a uniform way compatible with the change of dimensions~$N$ and~$2M$. In particular, with an appropriate choice of notation we make no distinction between the cases when $N$ is odd or even. As a corollary, we obtain a relation expressing odd Casimirs in terms of even ones, and this relation is universal in the sense that it does not depend on the dimensions~$N$ and~$2M$.

In all studied cases $\fg=\fso(N)$, $\fsp(2M)$ or $\fosp(N|2M)$ the corresponding universal enveloping algebra $U\fg$ admits the direct sum decomposition
$$
U\fg=(\mathfrak{n}_-\;U\fg+U\fg\; \mathfrak{n}_+)\oplus U\mathfrak{h},
$$
where $\mathfrak{n}_-$, $\mathfrak{h}$ and $\mathfrak{n}_+$ are the subalgebras of~$\fg$ spanned by $X_{ij}$ with $i>j$, $i=j$, and $i<j$, respectively. The Harish-Chandra homomorphism can be defined as the projection to the second summand of this decomposition. Its restriction to the center $ZU\fg$ is injective which allows one to identify this center with certain subalgebra in the algebra of polynomials in the generators $X_{i,i}$. The Perelomov-Popov formula describes the Casimir elements through this identification. It expresses the Casimir elements as (super)symmetric polynomials in the squares of the properly shifted generators $X_{i,i}$. More explicitly, let us collect Casimir elements into the following generating series in (the inverse powers of) the variable $u$
$$
F(u)=1-\Bigl(\frac{u-\frac{C_0-1}{2}}{u-\frac{C_0-2}{2}}\Bigr)\;\sum_{m=0}^\infty C_m u^{-m-1}.
$$
Then the analog of Perelomov-Popov formula for this case derived in~\cite{NR} can be represented in the form
\begin{equation}\label{eq:PP}
F(u)=\frac{P\bigl(u-\frac{C_0}{2}\bigr)}{P\bigl(u-\frac{C_0-2}{2}\bigr)},
\end{equation}
where $P$ is a rational function given by an explicit formula depending on the chosen algebra $\fg$.

\paragraph{$\fg=\fso(N)$:}  Set $x_i=X_{i,i}+\frac{N}{2}-i$, $i=1,\dots,[N/2]$.
Then we have
$$
P(u)=\begin{cases} \phantom{u\,}\prod_{i=1}^{\frac{N}{2}}(u^2-x_i^2),&N\text{ is even},\\[1ex]
 u\,\prod_{i=1}^{\frac{N-1}{2}}(u^2-x_i^2),&N\text{ is odd}.
\end{cases}
$$

\paragraph{$\fg=\fsp(2M)$:}  Set $x_i=X_{i,i}-i+M+1$. $i=1,\dots,M$,
Then we have
$$
P(u)=\frac{1}{\prod_{i=1}^{M}(u^2-x_i^2)}.
$$

\paragraph{$\fg=\fosp(N|2M)$:}  Set $\sigma_i=(-1)^{\eta_i}$ and $x_i=X_{i,i}+r_i$, $i=1,\dots,M+[N/2]$, where
$$
r_i=\sigma _i\left(\frac{1}{2} (N-2 M)-\sum _{j=1}^{i-1} \sigma _j\right)-\frac{1}{2} \left(\sigma _i+1\right).
$$
Then we have
$$
P(u)=\begin{cases} \phantom{u\,}\prod_{i=1}^{M+\frac{N}{2}}(u^2-x_i^2)^{\sigma_i},&N\text{ is even},\\[1ex]
 u\,\prod_{i=1}^{M+\frac{N-1}{2}}(u^2-x_i^2)^{\sigma_i},&N\text{ is odd}.
\end{cases}
$$

In all the cases the function $P(u)$ is either even or odd so that $P(u)/P(-u)=\pm 1$. This implies the following identity. 

\begin{proposition}
 There are universal polynomial expressions independent of~$N$ and $2M$ for odd Casimir elements of the Lie superalgebra $\fosp(N|2M)$ in terms of even ones. More explicitly, these relations can be expressed as the following equality for the generating series for Casimirs    
$$
F(u)F(C_0-1-u)=1.
$$
\end{proposition}

The first few of these relations are present in Remark~\ref{rem:Casimirs}.


\begin{thebibliography}{9}
\bibitem{bar1995vassiliev}
Dror Bar-Natan,
\newblock \emph{On the Vassiliev knot invariants},
\newblock Topology, 34(2):423--472, 1995.
\newblock (an updated version available at \url{http://www.math.toronto.edu/~drorbn/papers}).

\bibitem{Chmutov2012}
S.~Chmutov, S.~Duzhin, and J.~Mostovoy,
\newblock {\em Introduction to {Vassiliev Knot Invariants}},
\newblock Cambridge University Press, May 2012.

\bibitem{FKV97}
Figueroa-O'Farrill, T.~Kimura, A.~Vaintrob,
\newblock {\it The Universal Vassiliev Invariant for the Lie Superalgebra $\fgl(1|1)$}
\newblock Comm. Math. Phys.,
\newblock 1997, 185, 93--127

\bibitem{KL} M.~Kazarian, S.~Lando, Weight systems and invariants of graphs and embedded graphs, Russian Mathematical Surways, 2022, 77(5)

\bibitem{KY} M.~Kazarian, Zhuoke Yang, Extended weight systems on permutations associated with the Lie superalgebras $\fp(N)$ and $\fq(N)$, \emph{in preparation}.

\bibitem{kon1993}
M.~Kontsevich,
\newblock \emph{Vassiliev knot invariants},
\newblock in: Advances in Soviet Math., 16(2):137--150, 1993.

\bibitem{lando2013graphs}
Sergei~K. Lando and Alexander~K. Zvonkin.
\newblock {\em Graphs on surfaces and their applications}, volume 141.
\newblock Springer Science \& Business Media, 2013.

\bibitem{Perelomov-Popov}
Perelomov,~Askol'D.~Mikhailovich, and Vladimir~S.~Popov. 
\newblock {\em Casimir operators for semisimple Lie groups}
\newblock Mathematics of the USSR-Izvestiya 2.6 (1968): 1313

\bibitem{nwachuku1979}
Nwachuku C O.
\newblock{\em New expressions for the eigenvalues of the invariant operators of O(N) and Sp(2n)}
\newblock Journal of Mathematical Physics, 1979, 20(6): 1260-1266.

\bibitem{NR} C.O.Nwachuku, M.A. Rashid, New expressions for the eigenvalues of the invariant operators of the general linear and orthosymplectic Lie superalgebras, J. Math. Phys. 26, 1914 (1985)

\bibitem{Y1} Zhuoke Yang,
\newblock {\it New approaches to $\fgl(N)$ weight system},
arXiv:2202.12225 (2022)

\bibitem{Y2} 
Zhuoke Yang,
\newblock \emph{On the Lie superalgebra $\fgl(N|N)$ weight system},
\newblock {Journal of Geometry and Physics. 2023. Vol. 187. Article 104808.}

\bibitem{Zh} 
Zhelobenko~D.P.
\newblock \emph{Compact Lie groups and their representations},
\newblock Nauka, Moscow, 1970.
\newblock English translation: Translations of mathematical monographs, v. 40.
\newblock American Mathematical Society, Providence, Rhode Island, 1973.

\end{thebibliography}
\end{document}